\documentclass[12pt,times]{article}  
\usepackage{amsmath, bm}
\usepackage{amssymb}
\usepackage{placeins}
\usepackage{textcomp}
\usepackage{multirow}
\usepackage{amsthm}
\usepackage{color}
\usepackage{tipa}
\usepackage{graphicx}
\usepackage{psfrag}
\usepackage{amsfonts}
\usepackage{mathrsfs}
\usepackage{epstopdf}
\usepackage{epsfig}
\usepackage{pbox}
\usepackage{tabularx, caption}

\usepackage{amsmath}
\usepackage{algorithm}
\usepackage[noend]{algpseudocode}
\usepackage{algpseudocode}
\usepackage{pifont}
\usepackage{comment}

\usepackage{lineno}

\makeatletter
\def\BState{\State\hskip-\ALG@thistlm}
\makeatother

\def\mysection{\setcounter{equation}{0}\section}

\newcommand{\proofend}{\mbox{ }\hfill \raisebox{.4ex}{\framebox[1ex]{}}}

\newtheorem{Theorem}{Theorem}[section]
\newtheorem{Lemma}[Theorem]{Lemma}

\newtheorem{Remark}[Theorem]{Remark}

\newcommand {\p}{\partial}
\numberwithin{equation}{section}
\begin{document}
\title{\Large A finite element discretization with semi-implicit nonlinear multistep scheme for a two-dimensional competition-diffusion system of three competing species with different mobility rates}
\author{\small Xianping Li\footnotemark \ \  Woinshet D. Mergia\footnotemark \ \ and \ \ Kailash C. Patidar\footnotemark} 
\date{}
\maketitle
\def\thefootnote{\fnsymbol{footnote}}
\setcounter{footnote}{0}
\footnotetext[1]{College of Integrative Sciences and Arts, Arizona State University, Mesa, AZ, USA. E-mail: xianping@asu.edu}
\footnotetext[2]{Department of Mathematics, Jimma University, Jimma, Ethiopia. E-mail: woinshetdefar20@gmail.com}
\footnotetext[3]{Department of Mathematics and Applied Mathematics, University of the Western Cape, Private Bag X17, Bellville 7535, South Africa. E-mail: kpatidar@uwc.ac.za.}
\noindent
\begin{abstract}
\noindent
In ecological studies of pattern formation, models of the competitive-diffusion type are generally singularly perturbed, and the numerical approximation of such models is challenging. 
In this paper, we present finite element discretization combined with a second-order semi-implicit nonlinear multistep scheme for a two-dimensional three-species competition-diffusion system with distinct mobility rates. 
The method employs a $C^0$-conforming Galerkin finite element approximation in space and a Crank-Nicolson/Adams-Bashforth-type time integration that treats the diffusion terms implicitly while linearizing the nonlinear reaction terms in a stage-by-stage manner. 
The resulting scheme is linear at each time step and avoids iterative nonlinear solvers. Rigorous stability analysis shows that the discrete method inherits the asymptotic stability properties of the continuous model without restrictions on the time step size. 
Numerical simulations for various mobility regimes demonstrate the ability of the proposed method to capture complex spatio-temporal patterns, including droplet-like, banded, spiral, and glider-type structures.
\end{abstract}


\mysection{Introduction}\label{intro}
The study of spatio-temporal pattern formation in ecologically interacting species has been an intense research area in the last few decades. Some seminal works in this field are accounted for famous ecologists and mathematicians such as Fisher \cite{fisher1937wave}, 
Kolomogrov et al. \cite{kolmogorov1937study} and Volterra \cite{Voltera1926}. It is still continuing attracting enormous research interest, see, e.g.,  \cite{murray1989, okubo1980, Pereira-2018, shigesada1997biological, Tandon-2016}.

Mathematical modelling and modern computing facilities play an important role in the field of mathematical ecology. One way to mathematically describe the interaction of two or more species, such as competing for resources such as food and territory, 
is by using competitive-diffusion models. Such models have been intensively investigated by several researchers. For example, in \cite{contento2015, ei1999segregating, mimura2015dynamic}, a three-species Lotka-Volterra competition-diffusion model with 
the involved species having equal and very small mobility rates compared to the magnitude of the growth terms was used to demonstrate some new types of spatio-temporal patterns that have important applications in biodiversity. 
However, in nature, it is evident that the majority of such interactions of competing species disperse at different rates. Furthermore, previous studies have shown that differences in the dispersal distribution of competing species can fundamentally alter the outcome of competition \cite{petrovskii2001}.
In this study, we consider the numerical simulation of the Lotka-Volterra cyclic competition-diffusion model with different mobility rates. The governing model is described by a system of time-dependent partial differential equations of the reaction-diffusion type (\cite{petrovskii2001}). 
The model problem exhibits a new dynamical behavior with cyclic competition and is different from the existing patterns formed from reaction-diffusion systems used previously \cite{cangiani2017revealing, petrovskii2001}.

For problems that can be additively split into linear and nonlinear parts, implicit-explicit (IMEX) methods, in which the linear part is treated implicitly while the nonlinear part is treated explicitly, have been among the most popular choices in the research community. 
Various IMEX schemes have been developed to address the reaction-diffusion problem in different fields of study. Boscarino {et \,al.} \cite{boscarino2016high} developed higher-order semi-implicit IMEX schemes based on the Runge-Kutta method for a general class of problems 
with stiff and non-stiff additive parts, including reaction-diffusion problems. Cai and Cen \cite{cai2007numerical} proposed a slightly different second-order convergent scheme based on stretched variables for the singularly perturbed two species predator-prey model in two space dimensions. 
In \cite{cangiani2017revealing}, an IMEX--type scheme with automatic spatial adaptive capability was developed to study various types of interactive patterns in cyclic competition-diffusion systems of three species with different mobility rates. 
They applied the explicit Adams-Bashforth method for the nonlinear reaction term and the implicit Crank-Nicolson method for the diffusion term.

The numerical method presented in this paper considers a $C^0$-conforming finite element method for discretization in space,
which leads to a large nonlinear system of ODEs. It was then temporally discretized using a high-order stable finite difference method. Owing to the relative dominance of the reaction rates over the diffusion rates, the commonly used temporal integration IMEX methods lead to schemes with very weakened stability behavior. 
This is because of the fact that, in such IMEX methods, the diffusion (the less dominant part) is treated implicitly, whereas the reaction terms (the dominant part) is treated explicitly, thus leading to a higher tendency of the global scheme to behave as an explicit scheme. 
To overcome the restrictive nature of IMEX methods, in addition to treating the diffusion term implicitly, we use a semi-implicit method based on the Crank-Nicolson and Adams-Bashforth methods for the reaction term, whereas the diffusion terms are treated implicitly using the Crank-Nicolson method. 
The resulting scheme is multistep and nonlinear with respect to the previous time-steps. This leads to one of the desirable properties in time-stepping algorithms, referred to as {\em asymptotic consistency}, in which the discrete problem replicates the asymptotic stability behavior of the continuous problem. 
Note that the scheme is implicitly linear with respect to the current unknown step; hence, no iterative nonlinear solver is required. In addition, we demonstrate the existence of several new complex regular spatio-temporal patterns in two-dimensions, which may have important biological applications and are not available in the literature. 
Note that the cyclic competition of the three species without diffusion leads to the extinction of one or more species, as shown in \cite{gilpin1975limit, may1975}. However, the same model of cyclic competition with a diffusion term allows for the existence of all three species \cite{durrett1998spatial, petrovskii2001, reichenbach2007mobility}.

The remainder of this paper is organized as follows. In Section \ref{Three species-PP-6}, we provide an overview of the governing mathematical model and its stability analysis. In Section \ref{numerical method-PP-6}, the proposed numerical methods are formulated, and their asymptotic consistency is proved. 
Extensive numerical simulations demonstrating the performance of the numerical scheme are presented in Section \ref{numerical results-PP-6}. Finally, concluding remarks are presented in Section \ref{Concluding remarks-PP-6}.

\section{Three species competitive-diffusion model}\label{Three species-PP-6}
Let $\Omega$ be an open and bounded subset of $\mathbb{R}^2$, with piecewise smooth boundary $\partial \Omega$, representing the habitat of an ecosystem in which three competing species, whose population densities are denoted by $u_i$ ($i=1,2,3$), live in and interact with each other. 
Consider a three-species Lotka-Volterra competition-diffusion model describing the spatial interaction of the species \cite{may1975, petrovskii2001}:

\begin{eqnarray}\label{model-problem}
\dot{u}_i = \nabla\cdot(D_i\nabla u_i) + u_if_i(u_1, u_2, u_3);
~~~~~i=1,2,3,
\end{eqnarray}
where the superposed dot denotes time derivative and $\nabla$ is the gradient operator; $D_i$ ($i = 1,2,3$), which may possibly depend on position and time, is the (mobility) diffusivity of species $i$. The linear factor of the growth term is given by:
\begin{equation}\label{eq:growth-term}
f_i(u_1, u_2, u_3) =r_i-\sum_{j=1}^{3}a_{ij}u_j,~~~(i=1,2,3),
\end{equation}
where the parameters $r_i$ ($i=1,2,3$) is the intrinsic growth rate, and $a_{ij}$ is the inter-specific (when $i\neq j$) or intra-specific competition (self-limitation) rate. It is assumed that these parameters are positive. 
Systems \eqref{model-problem} and \eqref{eq:growth-term} are supplemented with homogeneous Neumann boundary conditions
\begin{equation}\label{eq:neumann-boundary}
\dfrac{\partial u_i}{\partial \bm{n}} = 0,~~~\text{for}~~\partial\Omega\times\mathbb{R}^+, ~~~~~i=1,2,3,
\end{equation}
and initial conditions
\begin{equation}\label{eq:initial}
u_1(\bm{x},0) = u_0(\bm{x}),~~~u_2(\bm{x},0) = v_0(\bm{x}),~~~ \text{and}~~u_3(\bm{x},0) = w_0(\bm{x}).
\end{equation}
Here, $\bm{x}$ represents the coordinate of a point in $\Omega$, $\bm{n}$ is a unit vector normal to the boundary $\partial \Omega$, and $u_0$, $v_0$, $w_0$ are some prescribed positive functions defined over the spatial domain $\Omega$.

Eq. \eqref{model-problem} has 15 parameters that could complicate the analysis, hence, it is instructive to scale the densities of each species according to their carrying capacities and then rescale the time and space coordinates. This gives
\begin{eqnarray}\label{rescale}
 \left.
 \begin{aligned}
a_{11} &= a_{22} = a_{33} = 1,~~ \varepsilon_1 = 1, \\
\varepsilon_{2} &= \frac{D_{2}}{D_{1}},~~\varepsilon_{3} = \frac{D_{3}}{D_{1}},~~ \tilde{x} = \sqrt{\frac{1}{D_{1}}}x,\\
a_{12} &= a_{23} = a_{31} = a,  ~~ a_{13} = a_{21} = a_{32} = b.
\end{aligned}
\right\}
\end{eqnarray}
Doing so, we obtain the following simplified competition-diffusion system
\begin{eqnarray}\label{scaled-model}
 \left.
 \begin{aligned}
\dot{u}_1& = ~\,~\Delta u_1 + u_1(1-u_1-a u_2-b u_3),\\
\dot{u}_2& = \varepsilon_2\Delta u_2 + u_2(1-b u_1- u_2-a u_3),\\
\dot{u}_3& = \varepsilon_3\Delta u_3 + u_3(1-\alpha u_1- b u_2- u_3),\\
\end{aligned}
\right.
\end{eqnarray}
where $\alpha=a_{31}$. Unless stated otherwise, $\alpha=a$. To make the dynamics richer we sometimes consider $\alpha \neq a$; hence, the total number of parameters become five. 
We also assume that $0<\varepsilon_{2}\, , \varepsilon_{3} \leq 1$.

Let us first consider the local stability behavior of model problem \eqref{scaled-model} without diffusion. In this case Eq. \eqref{scaled-model} gives

\begin{eqnarray}\label{spatialhomo}
\left.
\begin{aligned}
\dfrac{\mathrm{d}u_1}{\mathrm{d}t~} &= u_1(1-u_1- a u_2-b u_3):=g_1(u_1,u_2,u_3),\\
\dfrac{\mathrm{d}u_2}{\mathrm{d}t~} &= u_2(1-b u_1- u_2-a u_3):=g_2(u_1,u_2,u_3),\\
\dfrac{\mathrm{d}u_3}{\mathrm{d}t~} &= u_3(1-a u_1- b u_2- u_3):=g_3(u_1,u_2,u_3).
\end{aligned}
\right\}
\end{eqnarray}
We now analyze the local stability of the spatially homogeneous system \eqref{spatialhomo} because the dynamics of \eqref{scaled-model} are largely controlled by it. 
In addition, the analysis provides the necessary information on the choice of parameters for simulation. We obtain the equilibrium points of Eq. \eqref{spatialhomo} by equating the right-hand side to zero. 
This results in eight equilibria and the number of equilibrium points in the positive octant can differ for different parameter values. Furthermore, we assume the following
\begin{equation}\label{assume-parameter}
a + b > 2,\,\, a>1>b.
\end{equation}
Under the condition stipulated in Eq. \eqref{assume-parameter}, the system consists of exactly five equilibrium points in the positive quadrants
\[ A(0,0,0), B(0,0,1), C(0,1,0), D(1,0,0) ~~ \text{and}~~ E\left(\frac{1}{1+a+b},\frac{1}{1+a+b},\frac{1}{1+a+b}\right). \]

Following the standard procedure, we analyze the linear stability of these five equilibria by calculating the eigenvalues of the Jacobian matrix of the system in \eqref{spatialhomo} evaluated at the equilibria $(u_1^{*}, u_2^{*}, u_3^{*})$:
\begin{equation}
J = \begin{bmatrix}
 \frac{\partial g_1}{\partial u_1} & & \frac{\partial g_1}{\partial u_2} & & \frac{\partial g_1}{\partial u_3}  \\
 \\
 \frac{\partial g_2}{\partial u_1} & & \frac{\partial g_2}{\partial u_2} & & \frac{\partial g_2}{\partial u_3}  \\
 \\
 \frac{\partial g_{3}}{\partial u_1} & & \frac{\partial g_{3}}{\partial u_2} & & \frac{\partial g_{3}}{\partial u_3}  \\
\end{bmatrix}_{(u_1^{*}, u_2^{*}, u_3^{*})}.
\end{equation}
The trivial equilibrium state $A(0,0,0)$, with eigenvalues $\lambda_{1} = \lambda_{2}=\lambda_{3} = 1 $, is an unstable node describing the total extinction of the three species. 
The one-species equilibrium states $B(0, 0, 1)$, $C(0,1,0)$ and $D(1,0,0)$ have eigenvalues $\lambda_{1} = -1\,,~\lambda_{2} = 1-b$ and $\lambda_{3} = 1-a$; which implies that all of them are saddle points. 
These equilibrium points are biologically important in that they represent coexistence when all of them appear at the same time in different regions of a spatial habitat.

Steady state $E$ is the most interesting biological state because all three population densities coexist at $E$. At $E$ the corresponding Jacobian matrix becomes
\begin{equation}
J(E) = -\dfrac{1}{1+a+b}\begin{bmatrix}
 1 & a & b  \\
 b & 1 & a \\
 a & b & 1 \\
\end{bmatrix},
\end{equation}
and the associated eigenvalues are
\begin{eqnarray}\label{E-eigenvalue}
\begin{aligned}
\lambda_{1} &=-(1 +a +b),\\
\lambda_{2,3} &=-(1 - (a+b)/2) \pm i \left(\dfrac{\sqrt{3}}{2}\right)(a-b).\\
\end{aligned}
\end{eqnarray}
Hence, this equilibrium point is asymptotically stable if $a + b < 2$, neutrally stable if $a + b = 2$ and unstable if $a + b > 2$.

Under restriction \eqref{assume-parameter}, the only coexistence pattern is always a saddle point, and the only attractor in the phase space is a heterocyclic cycle consisting of three one-species equilibrium points.

The dynamical behaviors of complex patterns arising in the competition-diffusion model \eqref{model-problem} have been extensively studied, for example, 
\cite{adamson2012revising, chen2012exact, contento2015, kishimoto1982diffusive,  mimura1986, mimura2015dynamic,  petrovskii2001}. Coexistence in the form of cyclic competition patterns is an interesting aspect of such models for ecological and related applications. 
A cyclic competition pattern corresponds to a scenario, in relation to a family of parameters in which one species dominates the other in a cyclical manner, as depicted in Fig.~\ref{fig:cycle}. 
WHen each of these species has the same mobility rate, spiral waves occur at each triple junction (a point where the three distinct regions, each dominated by a single species, meet). 
However, when species have different mobility rates, the scenario is much more complex. One of the aims of this paper is to illustrate the capabilities of the numerical method, which will be discussed in Section~\ref{numerical method-PP-6}, 
in generating interesting complex spatio-temporal patterns that have been studied in \cite{ adamson2012revising, cangiani2017revealing, petrovskii2001}.
\begin{figure}[!h]
\begin{center}
\begin{tabular}{c}
\includegraphics[angle=0,width=6.5cm,height=5.0cm ]{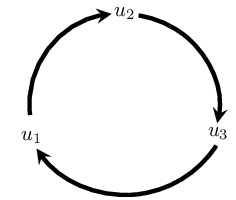}
\end{tabular}
\caption{{\small Schematic representation of cyclic competition. The arrows indicate the direction of domination.}}
\label{fig:cycle}
\end{center}
\end{figure}
Due to the presence of strong nonlinear coupling in the growth terms and the fact that the solutions of such a competitive-diffusion system may involve very fine moving spatial features as internal layers, 
the numerical approximation of such competitive-diffusion systems is very challenging. In what follows, we present a semi-implicit nonlinear multistep method based on a Method-of-Lines framework, 
in which $C^0$-conforming finite element is employed for spatial discretization.

\section{Numerical method}\label{numerical method-PP-6}
In this section, we present a high-order semi-implicit method for the competition-diffusion system of partial differential equations (Eqs. \eqref{model-problem}-\eqref{eq:initial}). 
First, we employ the Galerkin finite element method to discretize the problem in space, resulting in a large system of nonlinear ordinary differential equations. 
We then use a novel high-order semi-implicit scheme based on a multistep Lagrange method to integrate this system in a staggered manner.

\subsection{Spatial discretization: Galerkin finite element method}
The first step in the Galerkin finite element method for partial differential equations is to convert the strong form into a suitable equivalent integral equation known as a weak form. 
To this end, let $v$ be any test function in space $H^1(\Omega)$ (the set of functions on $\Omega$ whose derivatives are square integrable). 
We then multiply \eqref{model-problem} by $v$ and integrate it over the spatial domain $\Omega$. Then, after employing integration by parts together with the boundary conditions \eqref{eq:neumann-boundary}, 
we obtain the weak form that reads as follows: 
find $u_i\in H^1(\Omega)$, $i = 1,2,3$, such that
\begin{equation}\label{eq:weak-form}
\langle\dot{u}_i,~v\rangle= -\langle D_i\nabla u_i,~\nabla v\rangle + \langle u_if_i(u_1,u_2,u_3),~v\rangle,~~~~~\text{for every }~v\in H^1(\Omega),
\end{equation}
where $\langle\cdot,~\cdot\rangle$ denotes the integral-inner product between the two functions. The strong form \eqref{model-problem} and weak form \eqref{eq:weak-form} are equivalent in the sense that any sufficiently smooth solution of one is also a solution to the other.

Let $\bm{\mathcal{T}}=\lbrace\Omega^e\rbrace_{e=1}^{ne}$ be a triangulation of the spatial domain $\Omega$ into total $ne$ non-overlapping quadrilaterals that approximately covers $\Omega$.
Let $\lbrace\varphi_k\rbrace_{k=1}^{nd}$ be a set of basis (shape) functions on mesh $\bm{\mathcal{T}}$, where $nd$ is the total number of nodes in the mesh. Each shape function $\varphi_k$ has a local support over the elements that share the $k$-th node
and is a piecewise polynomial (usually Lagrange polynomials) whose degree is determined by the local number of nodes that a typical element $\Omega^e$ possesses. Moreover, they satisfy the interpolatary condition
\begin{equation}
\varphi_k(\bm{x}_j) = \delta_{kj};~~~~~~~~~~~~~~~~k, j=1, 2, \dots, nd,
\end{equation}
where $\delta_{kj}$ is the Kronecker delta function and $\bm{x}_j$ is the coordinate of the $j$-th node.

Now, we replace the population density function $u_i$ by the finite element interpolation $u_i^h$ which is defined by
\begin{equation}\label{eq:u-ih}
u_i^h(\bm{x}, t) = \sum_{k=1}^{nd}\varphi_k(\bm{x})u_{i[k]}(t),
\end{equation}
where $u_{i[k]}(t)$ is the $k$-th nodal value of $u_i^h$ at time $t$. We also replace test function $v$ with shape function $\varphi_j$. Then the discrete form of \eqref{eq:weak-form} becomes
For $i=1,2,3$:
\begin{equation}\label{eq:semi-discrete}
\langle\dot{u}_i^h,~\varphi_j\rangle = -\langle D_i\nabla u_i^h,~\nabla \varphi_j\rangle + \langle [u_i^hf_i(u_1^h,u_2^h,u_3^h)],~\varphi_j\rangle,~~~\text{for all } j=1,\dots,nd,
\end{equation}

For convenience, we rewrite Eq. \eqref{eq:semi-discrete} in matrix form as 
\begin{equation}\label{eq:semi-discrete-t}
\mathbf{M}\dot{\bm{u}}_i^h(t) = D_i\mathbf{K}\bm{u}_i^h(t) + \bm{b}_i(u_1^h,u_2^h,u_3^h),
\end{equation}
where $\bm{u}_i^h(t)$ is the vector of nodal values $u_i[k]$ at time $t$, and the entries of the mass and stiffness matrices are given by
\begin{equation}
\mathbf{M}_{jk} = \langle\varphi_j,~\varphi_k\rangle, ~~~\text{and}~
\mathbf{K}_{jk} = \langle\nabla\varphi_j,~\nabla\varphi_k\rangle,~~~~~~(j,k=1,\dots,nd).
\end{equation}
The right-hand side vectors $\bm{b}_i(u_1^h,u_2^h,u_3^h) = (b_j)$ $(i=1,2,3)$ are given by 
\begin{equation}\label{eq:b-j}
	b_j = \langle [u_i^h f_i(u_1^h,u_2^h,u_3^h)],~\varphi_j\rangle, ~~~ (j = 1,\dots,nd).
\end{equation}

By lumping (or multiplying the semi-discrete equations in \eqref{eq:semi-discrete} by $\mathbf{M}^{-1}$), and dropping the time dependence, we obtain:
\begin{equation}\label{eq:semi-discrete2}
	\dot{\bm{u}}_i^h = D_i\widehat{\mathbf{K}}\bm{u}_i^h + \mathbf{M}^{-1} \bm{b}_i(u_1^h,u_2^h,u_3^h),~~~\text{where }~\widehat{\mathbf{K}}=\mathbf{M}^{-1}\mathbf{K}, ~ i = 1, 2, 3.
\end{equation}

Next, we employ a multistep explicit-implicit method for the temporal integration. 

\subsection{Temporal integration}\label{sec:temporal}
Let $I=[0,~T]$ be the interval of interest and $\lbrace I_n\rbrace_{n=1}^{N}$ be the uniform partition of non-overlapping subintervals of the form $I_m = [t_{m-1},t_m]$, $m=1,2,3,\dots,N$ with step-size $\Delta t =t_{m}-t_{m-1}$.

We now discuss the temporal discretization of \eqref{eq:semi-discrete2} using a second-order semi-implicit technique involving the Crank-Nicolson and Adams-Bashforth methods. 
From an efficiency requirement point of view, we set up the semi-implicit scheme in the stage-by-stage (staggered) algorithm as follows.

Suppose we have the approximate solutions at $t_n$ and all the previous steps; in particular, we have $\lbrace \bm{u}_{i}^{n-1}\rbrace_{i=1}^{3}$ and $\lbrace \bm{u}_{i}^{n}\rbrace_{i=1}^{3}$. 
Then, the Crank-Nicolson and Adam-Bashforth methods read:
\begin{equation}\label{eq:u-hat}
\left.
\begin{aligned}
	\tilde{\bm{u}}_i^{n+1/2} &= -\frac{1}{2}\bm{u}_i^{n-1}+\frac{3}{2}\bm{u}_i^{n},\\
	\hat{\bm{u}}_i^{n+1/2} &= ~~\frac{1}{2}\bm{u}_i^{n~~}+\frac{1}{2}\bm{u}_i^{n+1},
\end{aligned}
\right\rbrace~~~~~~~(i=1,2,3).
\end{equation}
Note that $\tilde{\bm{u}}_i^{n+1/2}$ is computed explicitly, whereas $\hat{\bm{u}}_i^{n+1/2}$ involves an unknown value $\bm{u}_i^{n+1}$.  
We want to obtain the solution $\lbrace \bm{u}_{i}^{n+1}\rbrace_{i=1}^{3}$ at time $t_{n+1}$. 
To do so, we proceed as follows:
\begin{itemize}\label{eq:stage1}
\item[]{\bf Stage 1 ($i=1$):} For $b_j$ in \eqref{eq:b-j} with $i=1$, we approximate it as
\begin{eqnarray}
		\tilde{b}_j & = & \langle [\hat{{u}}_1^{n+1/2} f_1 \left(\tilde{{u}}_1^{n+1/2}, \tilde{{u}}_2^{n+1/2}, \tilde{{u}}_3^{n+1/2} \right)],~\varphi_j\rangle, ~~~ (j = 1,\dots,nd) \\
					& = & \int_{\Omega} \hat{{u}}_1^{n+1/2} f_1 \left(\tilde{{u}}_1^{n+1/2}, \tilde{{u}}_2^{n+1/2}, \tilde{{u}}_3^{n+1/2} \right) \varphi_j \, d\bm{x} \\
					& = & \int_{\Omega} \sum_{k=1}^{nd} \varphi_k(\bm{x})\hat{{u}}_{1[k]}^{n+1/2} f_1 \left(\tilde{{u}}_1^{n+1/2}, \tilde{{u}}_2^{n+1/2}, \tilde{{u}}_3^{n+1/2} \right) \varphi_j \, d\bm{x} \\
					& = & \sum_{k=1}^{nd} \left( \int_{\Omega} f_1 \left(\tilde{{u}}_1^{n+1/2}, \tilde{{u}}_2^{n+1/2}, \tilde{{u}}_3^{n+1/2} \right) \varphi_k \varphi_j \, d\bm{x} \right) \hat{{u}}_{1[k]}^{n+1/2}.
		\end{eqnarray}
Because $f_1$ is explicitly represented by known functions, we obtain the approximation for vector $\bm{b}_1$ as
\begin{equation}
	\tilde{\bm{b}}_1 = \mathbf{B} \hat{\bm{u}}_1^{n+1/2}, 
\end{equation}
where the entries of $\mathbf{B}$ are given by
\begin{equation}
	\mathbf{B}_{jk} = \int_{\Omega} f_1 \left(\tilde{{u}}_1^{n+1/2}, \tilde{{u}}_2^{n+1/2}, \tilde{{u}}_3^{n+1/2} \right) \varphi_k \varphi_j \, d\bm{x}
					:= \langle \tilde{f_1} \varphi_k, ~\varphi_j \rangle.
\end{equation}

Then, based on \eqref{eq:semi-discrete-t}, with $\bm{u}_1^h(t)$ approximated by $\hat{\bm{u}}_1^{n+1/2}$, we obtain $\bm{u}_1^{n+1}$ by solving
\begin{equation}
	\mathbf{M} \dfrac{\bm{u}_1^{n+1}-\bm{u}_1^{n}}{\Delta t} = D_1 {\mathbf{K}}\hat{\bm{u}}_1^{n+1/2} + \mathbf{B} \hat{\bm{u}}_1^{n+1/2}.
\end{equation}
Applying \eqref{eq:u-hat}, we have
\begin{equation}\label{eq:stage1}
	\left( \mathbf{M}-\frac{1}{2} \Delta t (D_1 \mathbf{K} + \mathbf{B}) \right) \bm{u}_1^{n+1} = \left( \mathbf{M}+\frac{1}{2} \Delta t (D_1 \mathbf{K} + \mathbf{B}) \right) \bm{u}_1^{n}.
\end{equation}

\item[]{\bf Stage 2 ($i=2$):} Having obtained $\bm{u}_1^{n+1}$,  we can compute $\hat{\bm{u}}_1^{n+1/2}$ using \eqref{eq:u-hat} and approximate $b_j$ for $i=2$ as
\begin{equation}
	\tilde{b}_j = \langle [\hat{{u}}_2^{n+1/2} f_2(\hat{{u}}_1^{n+1/2}, \tilde{{u}}_2^{n+1/2}, \tilde{{u}}_3^{n+1/2})],~\varphi_j\rangle, ~~~ (j = 1,\dots,nd).
\end{equation}
Because $f_2$ is explicitly represented by known functions, we can follow the same procedures as in Stage 1 to approximate $\bm{b}_2$ as
\begin{equation}
	\tilde{\bm{b}}_2 = \mathbf{B} \hat{\bm{u}}_2^{n+1/2}, 
\end{equation}
where the entries of $\mathbf{B}$ are given by
\begin{equation}
	\mathbf{B}_{jk} = \langle f_2(\hat{{u}}_1^{n+1/2}, \tilde{{u}}_2^{n+1/2}, \tilde{{u}}_3^{n+1/2}) \varphi_k, ~\varphi_j \rangle := \langle \tilde{f_2} \varphi_k, ~\varphi_j \rangle.
\end{equation}
Then we obtain $\bm{u}_2^{n+1}$ by solving
\begin{equation}\label{eq:stage2}
	\left( \mathbf{M}-\frac{1}{2} \Delta t (D_1 \mathbf{K} + \mathbf{B}) \right) \bm{u}_2^{n+1} = \left( \mathbf{M}+\frac{1}{2} \Delta t (D_1 \mathbf{K} + \mathbf{B}) \right) \bm{u}_2^{n}.
\end{equation}

\item[]{\bf Stage 3 ($i=3$):} We now compute $\hat{\bm{u}}_2^{n+1/2}$ and approximate $b_j$ as
\begin{equation}
	\tilde{b}_j = \langle [\hat{{u}}_3^{n+1/2} f_3(\hat{{u}}_1^{n+1/2}, \hat{{u}}_2^{n+1/2}, \tilde{{u}}_3^{n+1/2})],~\varphi_j\rangle, ~~~ (j = 1,\dots,nd).
\end{equation}
Then we have the approximation for $\bm{b}_3$ as
\begin{equation}
	\tilde{\bm{b}}_3 = \mathbf{B} \hat{\bm{u}}_3^{n+1/2}, 
\end{equation}
where the entries of $\mathbf{B}$ are given by
\begin{equation}
	\mathbf{B}_{jk} = \langle f_3(\hat{{u}}_1^{n+1/2}, \hat{{u}}_2^{n+1/2}, \tilde{{u}}_3^{n+1/2}) \varphi_k, ~\varphi_j \rangle := \langle \tilde{f_3} \varphi_k, ~\varphi_j \rangle.
\end{equation}
Finally, we obtain $\bm{u}_3^{n+1}$ by solving
\begin{equation}\label{eq:stage3}
	\left( \mathbf{M}-\frac{1}{2} \Delta t (D_1 \mathbf{K} + \mathbf{B}) \right) \bm{u}_3^{n+1} = \left( \mathbf{M}+\frac{1}{2} \Delta t (D_1 \mathbf{K} + \mathbf{B}) \right) \bm{u}_3^{n}.
\end{equation}
\end{itemize}
From each of the above stages we obtain a complete solution $\lbrace \bm{u}_{i}^{n+1}\rbrace_{i=1}^{3}$ at the current time  $t_{n+1}$. 
The stage-by-stage algorithm mentioned above is semi-implicit and formally second-order convergent in time. 
Moreover, updating function $\tilde{f_i}$ at stages 2 and 3 based on the already known solutions from stages 1 and 2 may enhance the stability property without additional computational cost.


Next, we show that the weak form in \eqref{eq:weak-form} has a unique solution. We denote $V=H^1(\Omega)$ and assume $D_i$ and $f_i$ are sufficiently smooth. 
We also assume that the initial and boundary conditions meet the compatibility and smoothness requirements such that the original problem has a unique solution. 
We assume that the exact solution is bounded in both space and time. Then, we obtain the following result:	

\begin{Theorem}
\label{thm-1}
	Assume $D_i \ge 0$ and $f_i(u_1, u_2, u_3) \in L^\infty(\Omega)$ for $i=1,2,3$ are evaluated explicitly.
	Then for each $i=1,2,3$, the variational problem \eqref{eq:weak-form} admits a unique solution $u_i \in H^1(\Omega)$.
\end{Theorem}
\noindent
\textbf{Proof.} At a fixed time level, denote the bilinear form $a_i: V \times V \to \mathbb{R}$ as
	$$a_i(u,v) := <D_i \nabla u, \nabla v> - <\tilde{f}_i u, v>, \quad (i=1,2,3). $$
where $\tilde{f}_i = f_i (u_1, u_2, u_3)$ is treated as a known function owing to the explicit evaluation in the stage-by-stage scheme.

First, we show that $a_i(\cdot,\cdot)$ is continuous on $V \times V$. Since $\tilde{f}_i \in L^\infty(\Omega)$, we have
\begin{eqnarray*} 
\begin{aligned}
|a_i(u,v)| & = \left|\int_{\Omega} D_i \nabla u \cdot \nabla v \, d\bm{x} - \int_{\Omega} \tilde{f}_i u v \, d\bm{x} \right| \\
& \leq D_i \Vert\nabla u \Vert_{L^2(\Omega)} \Vert \nabla v \Vert_{L^2(\Omega)} + \Vert \tilde{f}_i \Vert_{L^{\infty}(\Omega)} \Vert u \Vert_{L^2(\Omega)} \Vert v \Vert_{L^2(\Omega)} \\
& \leq C  \Vert u \Vert_{H^1(\Omega)} \Vert v \Vert_{H^1(\Omega)}, \\
\end{aligned}
\end{eqnarray*}
where $C > 0$ is independent of $u$ and $v$. Thus, $a_i(\cdot,\cdot)$ is continuous on $V \times V$.

Next, we show that $a_i(\cdot,\cdot)$ satisfies a G\r{a}rding-type inequality on $V$. For any $u \in V$, we have that
\begin{eqnarray*} 
\begin{aligned}
|a_i(u,u)| & = \left|\int_{\Omega} D_i \nabla u \cdot \nabla u \, d\bm{x} - \int_{\Omega} \tilde{f}_i u^2 \, d\bm{x} \right| \\
& \geq D_i \Vert\nabla u\Vert_{L^2(\Omega)}^2 - \Vert \tilde{f}_i \Vert_{L^\infty(\Omega)} \Vert u\Vert_{L^2(\Omega)}^2. \\
\end{aligned}
\end{eqnarray*}
Hence, there exists a constant $C_0 > 0$ such that
\begin{equation}
|a_i(u,u)| + C_0 \Vert u \Vert_{L^2(\Omega)}^2 \geq D_i \Vert \nabla u \Vert_{L^2(\Omega)}^2, \quad \forall u \in V.
\end{equation}	
This indicates that $a_i(\cdot,\cdot)$ satisfies the G\r{a}rding-type inequality in $V$.

Because $a_i(\cdot,\cdot)$ is continuous and satisfies a G\r{a}rding-type inequality, and the right-hand side functional is continuous on $V$,
the standard results for elliptic problems with bounded lower-order terms (see Ciarlet \cite{Ciarlet-1978}) imply that the variational problem \eqref{eq:weak-form} admits a unique solution 
$u_i \in H^1(\Omega)$ for each $i=1,2,3$.

\proofend

\begin{Theorem}
Let $V_h \subset V$ be a conforming finite dimensional subspace associated with regular triangulation $\mathcal{T}_h$. Assume $D_i \ge 0$ and $f_i(u_1, u_2, u_3) \in L^\infty(\Omega)$ ($i=1,2,3$) are evaluated explicitly.
Let $u_i \in V$ be the solution of \eqref{eq:weak-form} and let $u_i^h \in V_h$ be its Galerkin approximation defined by 
$$ a_i (u_i^h, v_h) = l_i(v_h), \quad \forall v_h \in V_h, $$
where $a_i(\cdot,\cdot)$ is the bilinear form defined in Theorem \ref{thm-1} and $l_i(\cdot)$ is the corresponding linear functional.
Then, there exists a constant $C > 0$ independent of $h$ such that
\begin{equation}
	\Vert u_i - u_i^h \Vert_{H^1(\Omega)} \le C \inf\limits_{v_h \in V_h} \Vert u_i - v_h \Vert_{H^1(\Omega)}.
\end{equation}
\end{Theorem}
\noindent
\textbf{Proof.} From Theorem \ref{thm-1}, the bilinear form $a_i(\cdot,\cdot)$ is continuous and satisfies a G\r{a}rding-type inequality on $V \times V$. 
Consequently, both the continuous and discrete variational problems are well posed.
Standard results on Galerkin approximations for elliptic problems with bounded lower-order terms (see Ciarlet \cite{Ciarlet-1978}) imply a quasi-optimal error estimate:
\begin{equation}
	\Vert u_i - u_i^h \Vert_{H^1(\Omega)} \le C \inf\limits_{v_h \in V_h} \Vert u_i - v_h \Vert_{H^1(\Omega)},
\end{equation}
where the constant $C$ depends only on the continuity and G\r{a}rding-type inequality constants of $a_i(\cdot,\cdot)$, but is independent of the mesh size $h$.
This completes this proof. 

\proofend

Next, we derive the a priori error estimate for the semi-discretized variational problem \eqref{eq:semi-discrete-t}.
Because $H^1$ functions are not continuous in two dimensions, we first consider $u_i \in H^{k+1}(\Omega)$ for some $k \ge 1$. 
For convenience, we ignore the time dependence of the variables and use $u$ to denote any of the $u_i$ $(i=1,2,3)$.

Denote by $P_k$ and $Q_k$ the space of polynomials of total and partial degrees less than or equal to $k$, respectively. 
Then $P_k \subset Q_k$ and $Q_k \subset P_{2k}$.
Consider a regular family of simplicial meshes $\{ \mathcal{T}_h \}$, where $h$ denotes the mesh size defined as $h:= \max\limits_{K \in \mathcal{T}_h} \text{diam}(K)$.
For any element $K \in \mathcal{T}_h$, denotes $\rho(K)$ the diameter of the inscribed circle. The mesh is regular, meaning that there exists a constant $\sigma>0$ such that
$\max\limits_{K \in \mathcal{T}_h}\frac{h(K)}{\rho(K)} \le \sigma$. Assume all the finite elements $(K,P_K,\Sigma_K)$, $K \in \cup_h \left\{\mathcal{T}_h\right\}$, 
are affine-equivalent to a single reference finite element $(\hat{K},\hat{P},\hat{\Sigma})$.

Denote $V_h \subset V$ the linear finite element space associated with mesh $\mathcal{T}_h$. Then, we have the following lemma.

\begin{Lemma} (\cite{Ciarlet-1978}): 
	There exists a constant $C$ independent of $h$ such that for any vector function $v \in H^{k+1}(\Omega)$,
	\begin{equation}
		\label{err1}
		\Vert v - \Pi_h v \Vert_{H^1(\Omega)} \le C h^k |v|_{H^{k+1}(\Omega)},
	\end{equation}
	where $\Pi_h$ denotes the $V_h$-interpolation operator.
\end{Lemma}

Let $u^h$ be an approximation of $u$ for $V_h$. Then, we obtain the following results:

\begin{Theorem}
\label{thm-2}
	There exists a constant $C$ independent of $h$ such that, if $u \in H^{k+1}(\Omega)$, we have
	\begin{equation}
		\label{err2}
		\Vert u - u^h \Vert_{H^1(\Omega)} \le C h^k | u |_{H^{k+1}(\Omega)}.
	\end{equation}
\end{Theorem}
\noindent
\textbf{Proof.} By C\'ea's Lemma, we have
\begin{eqnarray*} 
\begin{aligned}
\Vert u - u^h \Vert_{H^1(\Omega)} & \le C \inf\limits_{v^h \in V_h} \Vert u - v^h \Vert_{H^1(\Omega)} \\
& \le C \Vert u - \Pi_h u \Vert_{H^1(\Omega)}.
\end{aligned}
\end{eqnarray*}
Then, \eqref{err2} follows from \eqref{err1}.

\proofend

In terms of $L^2$-norm, we have:

\begin{Theorem}
	There exists a constant $C$ independent of $h$ such that, if $u \in H^{k+1}(\Omega)$, we have
	\begin{equation}
		\label{err3}
		\Vert u - u^h \Vert_{L^2(\Omega)} \le C h^{k+1} \sum_{|m|=k+1} \Vert \p^m u \Vert_{L^2(\Omega)},
	\end{equation}
where $\p^m u$ is the $m$-th order partial derivative.
\end{Theorem}
\noindent
\textbf{Proof.} The result follows from the Aubin-Nitsche Lemma and Theorem \ref{thm-2}. Further details can be found in Ciarlet (\cite{Ciarlet-1978}).

\proofend

\begin{Remark}
The results have been established under the assumptions that the solution $u$ is sufficiently smooth (in $H^{k+1}(\Omega)$ for some $k \ge 1$) and that the $V_h$-interpolant $\Pi_h u$ exists. However, it is still possible to prove the convergence of the method if the solution $u$ belongs to $H^1(\Omega)$ and ``minimal'' assumptions hold. See Ciarlet (\cite{Ciarlet-1978}) for more details.
\end{Remark}


\subsection{Stability analysis}

Numerical approximation of competitive-diffusion problems is challenging. This is particularly the case in the ecological application of competition-diffusion models, 
where spatial mobility (diffusion) is typically very small compared to the rate of interactions between the species. Such problems are characterized by the presence of a 
very small spatial scale as internal layers. In such cases, the stability of the associated scheme(s) is one of the most critical issues in the numerical approximation of their solutions.

The implicit-explicit (IMEX) schemes use a higher-order explicit scheme for the reaction terms and an implicit scheme of the same order for the linear diffusion terms. 
This usually leads to weakened stability behavior because the dominant term (reaction term) is explicit. Hence, in addition to treating the diffusion terms implicitly, 
the stage-by-stage method outlined in Section~\ref{sec:temporal} enhances the stability property using a semi-linearized but high-order scheme based on the Crank-Nicolson 
and Adams-Bashforth methods for the reaction term.

It is essential to analyze how the discretization of the reaction term affects the stability of the proposed numerical scheme. In particular, we analyze how the discrete scheme based on 
the Crank-Nicolson and Adams-Bashforth methods imitates the behavior of a continuous problem. To this end, we consider the following reduced scalar problem, obtained by ignoring the $u_2, u_3$ and other diffusion terms
\begin{equation}\label{eq:simplified-model}
\dot{u} = \lambda u(1 - u).
\end{equation}
For simplicity, we also assume that $r_1 = a_{11}=-\lambda > 0$, and drop the subscript 1. Scalar problem \eqref{eq:simplified-model} has $u = 0$ and $u=1$ as equilibrium solutions, 
from which only the former is asymptotically stable. Applying the Crank-Nicolson and Adams-Bashforth semi-linearized scheme to the reduced model \eqref{eq:simplified-model} leads to:
\begin{equation}\label{eq:simple-scheme}
u^{n+1} = G(u^{n-1},u^n;z),
\end{equation}
where
\begin{equation}\label{eq:rhs-simple-scheme}
G(u^{n-1},u^n;z) = \dfrac{1+\frac{1}{2}z(1+\frac{1}{2}u^{n-1}-\frac{3}{2}u^{n})}{1-\frac{1}{2}z(1+\frac{1}{2}u^{n-1}-\frac{3}{2}u^{n})}u^{n},~~~~ z = \Delta t\,\lambda.
\end{equation}
A discrete analogy of the equilibrium points of the continuous model \eqref{eq:simplified-model} is a fixed point of its discrete version \eqref{eq:simple-scheme}; that is, 
find $u^*$ satisfying
\begin{equation}
u^* = G(u^*,u^*;z).
\end{equation}
Hence, one of the consistent implications of the discrete scheme is that any equilibrium point is also a fixed point. 
Remarkably, this is verified because $u^*=0$ and $u^*=1$ are also fixed points of the discrete problem.

What remains now is to show that, in the discrete sense, these fixed points have local stability properties similar to those of the equilibrium points in the continuous case. 
To do this, we proceed as follows. Because Eq. \eqref{eq:simple-scheme} involves nonlinear terms in $u^{n-1}$ and $u^n$, we take the partial derivatives of the right-hand-side 
of \eqref{eq:rhs-simple-scheme} and evaluate them at the fixed points. This gives us a locally linearized discrete system:
\begin{equation}\label{eq:local-linear-discrete}
u^{n+1}=\dfrac{\partial G~~~}{\partial u^{n-1}}\bigg\vert_{u^{n-1}=u^n=u^*}\,u^{n-1}~~~+~~~\dfrac{\partial G}{\partial u^{n}}\bigg\vert_{u^{n-1}=u^n=u^*}\,u^{n}.
\end{equation}
For $u^*=0$, \eqref{eq:local-linear-discrete} becomes
\begin{equation}\label{eq:first-linearized}
u^{n+1}= \dfrac{2+z}{2-z}u^n,
\end{equation}
whereas for $u^*=1$, it becomes
\begin{equation}\label{eq:second-linearized}
u^{n+1}= \frac{1}{2}zu^{n-1}+ (1-\frac{3}{2}z)u^n.
\end{equation}
Following the standard procedure, we replace $\xi^{k+1}$ for $u^{n+k}$, for $k=-1,0,1$; to obtain the following characteristic polynomials corresponding to 
linearized schemes \eqref{eq:first-linearized} and \eqref{eq:second-linearized}:
\begin{equation}
\Pi_{0}(\xi; z) = \xi^2-\dfrac{2+z}{2-z}\xi,~~\text{and}~~\Pi_{1}(\xi; z) = \xi^2-(1-\frac{3}{2}z)\xi - \frac{1}{2}z.
\end{equation}
The stability requirement is that the roots of the characteristic polynomials should lie inside the unit ball centered at the origin of the complex $\xi$-plane. 
Hence, as shown in Fig.~\ref{fig:stab}, the fixed point $u^*=0$ is unconditionally stable, whereas $u^*=1$ is unconditionally unstable as required.

In addition to the use of semi-linearization of the reaction term, updating each equation successively based on the current time-step solutions of the already 
solved ones at the same time further increases stability.
\begin{figure}[!h]
\begin{center}
\begin{tabular}{cc}
\includegraphics[angle=0,width=6.5cm,height=5.0cm ]{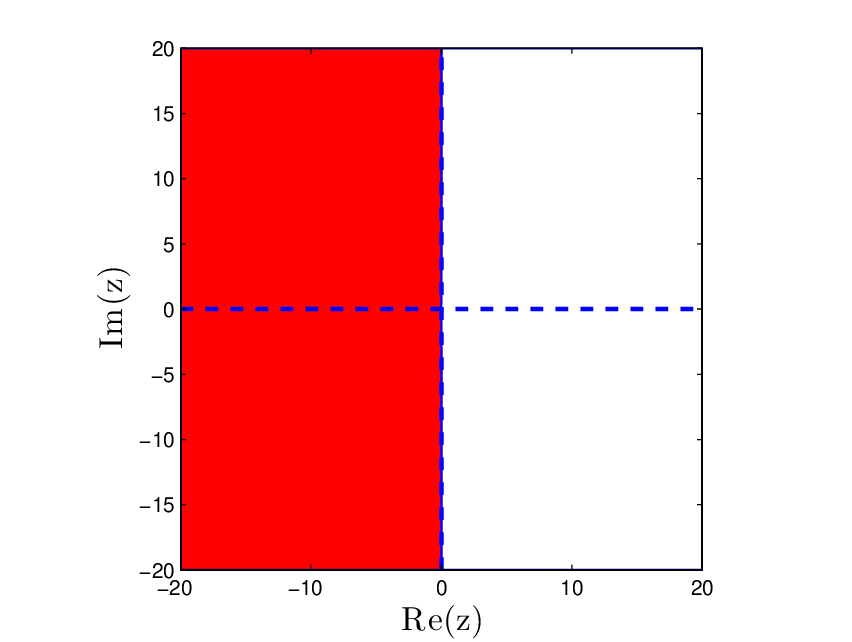}&
\includegraphics[angle=0, width=6.5cm,height=5.0cm]{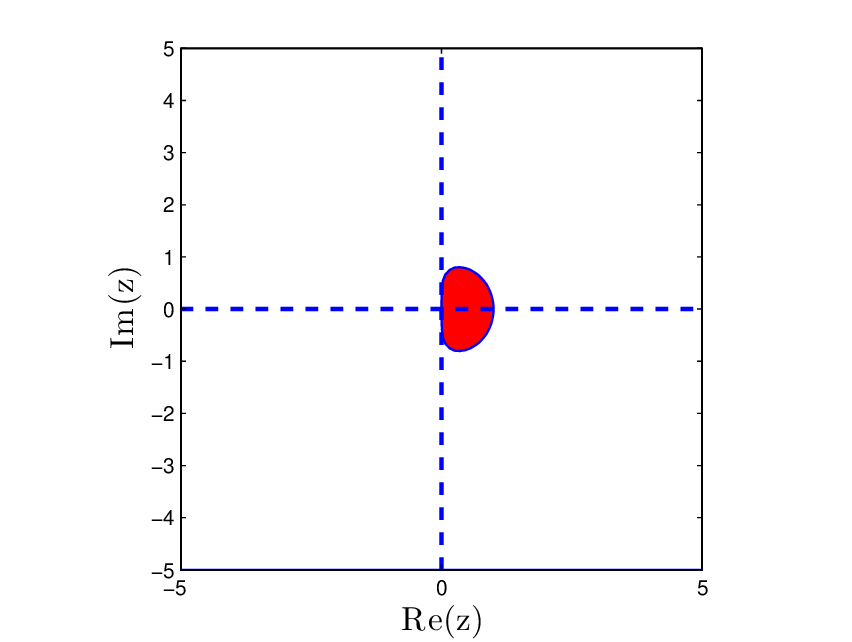}\\
(a) Local stability when $u^*=0$&(b) Local stability when $u^*=1$\\
\end{tabular}
\caption{{\small Stability regions of the semi-linearized scheme applied to the reduced nonlinear problem \eqref{eq:simplified-model}}.}
\label{fig:stab}
\end{center}
\end{figure}

\section{Numerical results}\label{numerical results-PP-6}
In this section, we present various numerical results to demonstrate the performance of our proposed schemes in simulating several cyclic competition patterns. 
The nonlinear complex spatio-temporal structures that were recently explored in \cite{cangiani2017revealing} are presented. These spatio-temporal patterns include 
droplet-like, band-like, glider-like, and regular spiral-like structures, each of which corresponds to a parameter set for which $a = 1,~ b = 2$ according to the 
rescaled system \eqref{rescale} and with different mobility rates. In all simulations, the spatial domain is considered as $\Omega \in [-2,~ 2] \times [-2,~ 2]$ with 
a mesh grid $251\times251$ of Q1 quadrilateral elements, except in the case of strip-like banded patterns, as shown in Fig.~\ref{fig:Band}, for which we 
considered $\Omega \in [-8,~ 8] \times [-8,~ 8]$ with a mesh grid $291\times291$ of Q1 quadrilateral elements. The time-step length used for all the simulations was $\Delta t = 1$. 
The initial condition is a simple segregation configuration with a single triple-junction at the top-right quarter of the domain, with a separation angle for each species of $2\pi/3$. 
The following color code was used throughout the simulations: green for $u_1$, blue for $u_2$, and red for $u_3$.

 We first considered the classical case of cyclic competition patterns, as shown in Fig.~\ref{fig:Droplet},~\ref{fig:Band}, and~\ref{fig:Spiral}. The shapes of the patterns differed depending 
 on the mobility rates of the species. As shown in Fig.~\ref{fig:Droplet}, where $\varepsilon_2 = 0.1$ and $\varepsilon_3 = 0.6$ were used, an expanding triangular droplet-like structures 
 with complex patterns were formed. A highly structured coexistence pattern of the three species emerged at the bottom left of the sharp wedge. The other type of coexistence pattern, 
 as shown in Fig.~\ref{fig:Band}, where $\varepsilon_2 = 0.1$ and $\varepsilon_3 = 0.9$, is a strip-like structure. It starts in the same way as Fig.~\ref{fig:Droplet} with the inside appearing 
 irregular coexistence pattern, but as time goes on, the shape diverges from the droplet-like pattern with the inside becoming more regular, which involves band-like structures at the bottom left 
 of the envelope. The number of bands increased and the envelope expanded as time increased. 
 
 Fig.~\ref{fig:Spiral} shows the simple spiral-like dynamical coexistence pattern with spiral centering at the initial triple junction. 
 In this case, the diffusion coefficients are equal, that is, $\varepsilon_2 = 1$ and $\varepsilon_3 = 1$. Finally, a conditional cyclic competition of the three species is displayed in Fig.~\ref{fig:Shooting} 
 for the parameter $\alpha = 1.3$ and diffusion coefficients $\varepsilon_2 = 0.55$ and $\varepsilon_3 = 0.5$. As shown in this figure, initially, it forms a spiral tip, from which the droplet shape detached later. 
 The dynamics moved to the left corner of the domain in which some of the structure persisted for a longer time and some disappear faster. 
 Finally, only one species survived longer as shown in the last sub-plot of Fig. \ref{fig:Shooting}.

\begin{figure}[!h]
\begin{center}
\begin{tabular}{cc}
\vspace*{1em}
\includegraphics[angle=0, width=6.0cm, trim={2cm, 0.5cm, 2cm, 0}, clip, height=6.0cm]{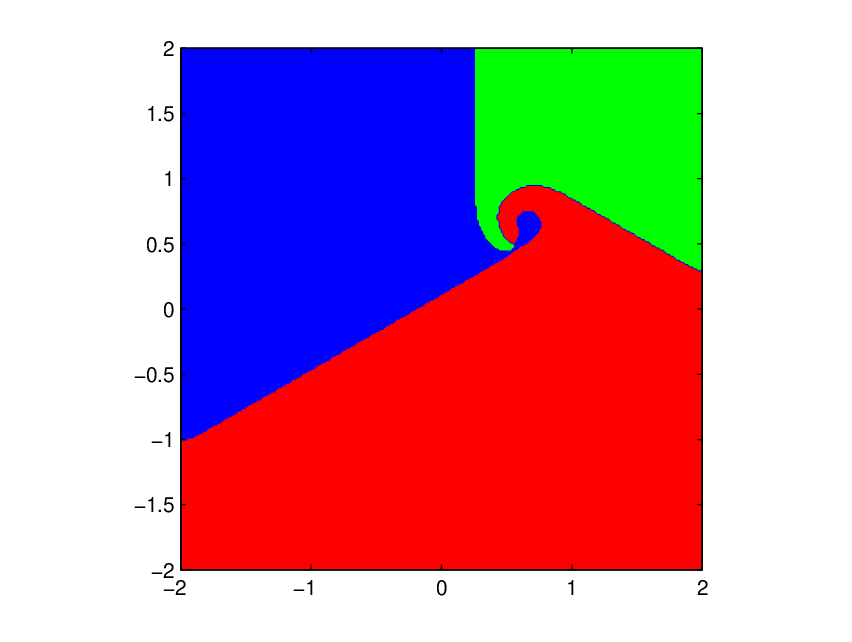}&
\includegraphics[angle=0, width=6.0cm, trim={2cm, 0.5cm, 2cm, 0}, clip, height=6.0cm]{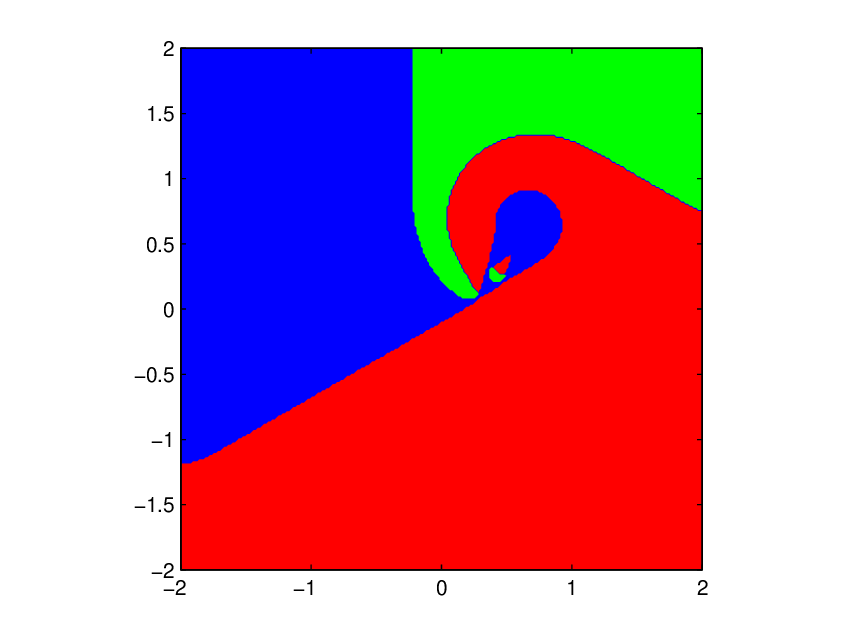}\\
$t = 80$ & $t = 150$\\
\includegraphics[angle=0, width=6.0cm, trim={2cm, 0.5cm, 2cm, 0}, clip, height=6.0cm]{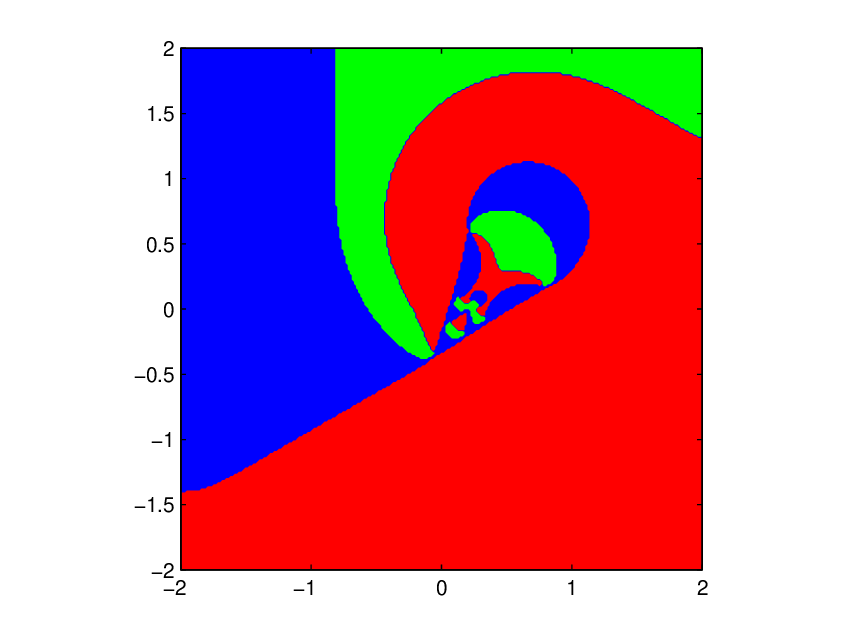}&
\includegraphics[angle=0, width=6.0cm, trim={2cm, 0.5cm, 2cm, 0}, clip, height=6.0cm]{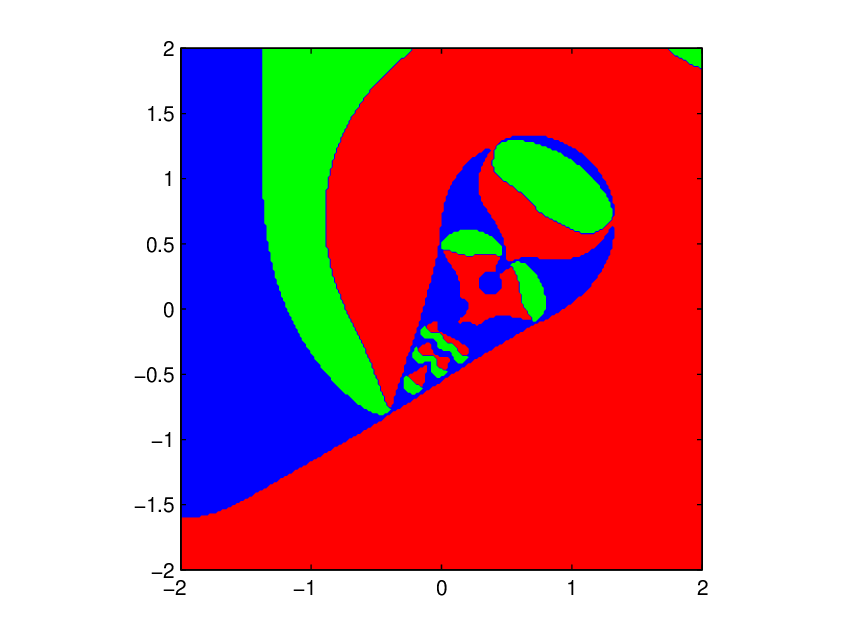}\\
$t =250$ & $t = 350$\\
\includegraphics[angle=0, width=6.0cm, trim={2cm, 0.5cm, 2cm, 0}, clip, height=6.0cm]{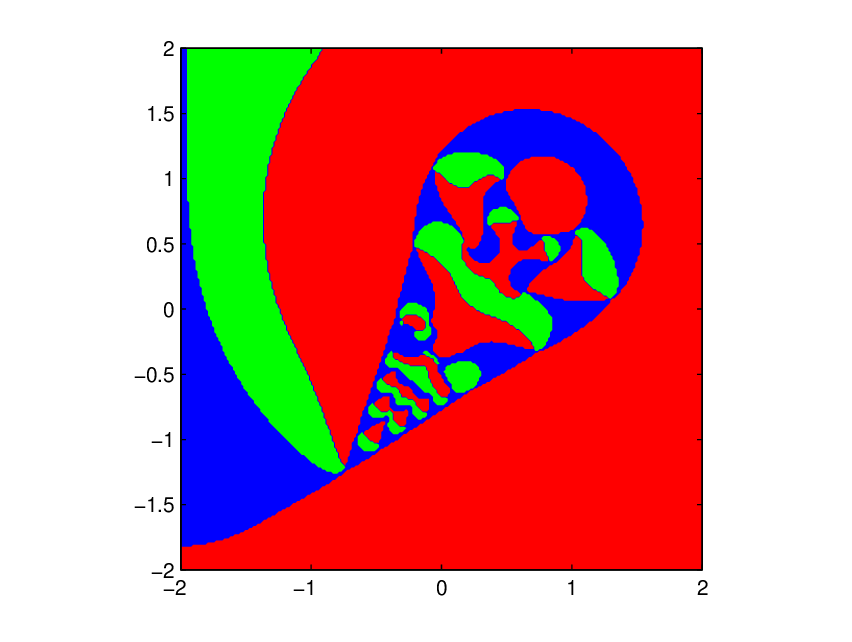}&
\includegraphics[angle=0, width=6.0cm, trim={2cm, 0.5cm, 2cm, 0}, clip, height=6.0cm]{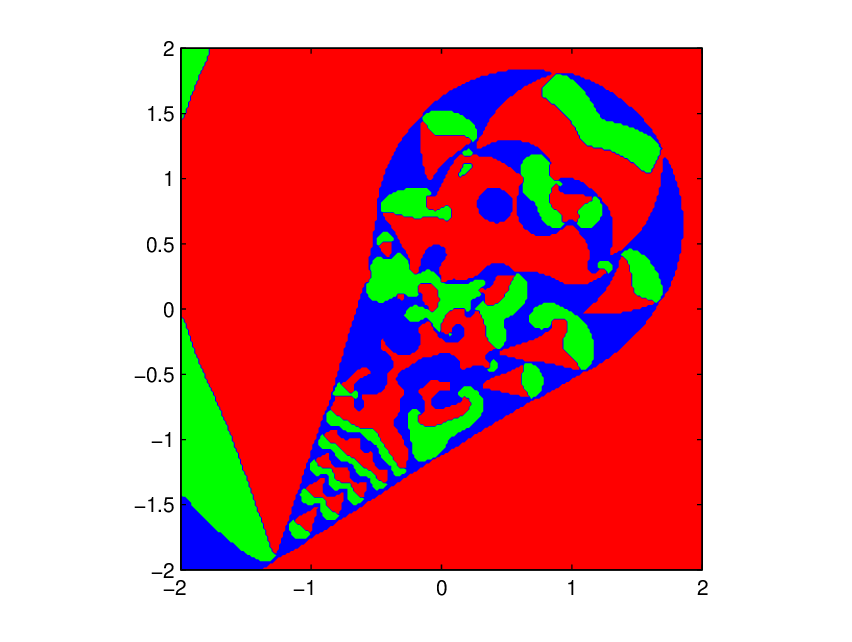}\\
$t =450$ & $t = 600$
\end{tabular}
\caption{{\small Droplet-like pattern in the dynamics of the three species at different times. Parameters used for the simulation are $a = 1 , ~ b = 2,~ \varepsilon_2 = 0.1,~ \varepsilon_3 = 0.6$, and $\Delta t=1$.}}
\label{fig:Droplet}
\end{center}
\end{figure}
\begin{figure}[!h]
\begin{center}
\begin{tabular}{cc}
\includegraphics[angle=0, width=6.0cm, trim={2cm, 0.5cm, 2cm, 0}, clip, height=6.0cm]{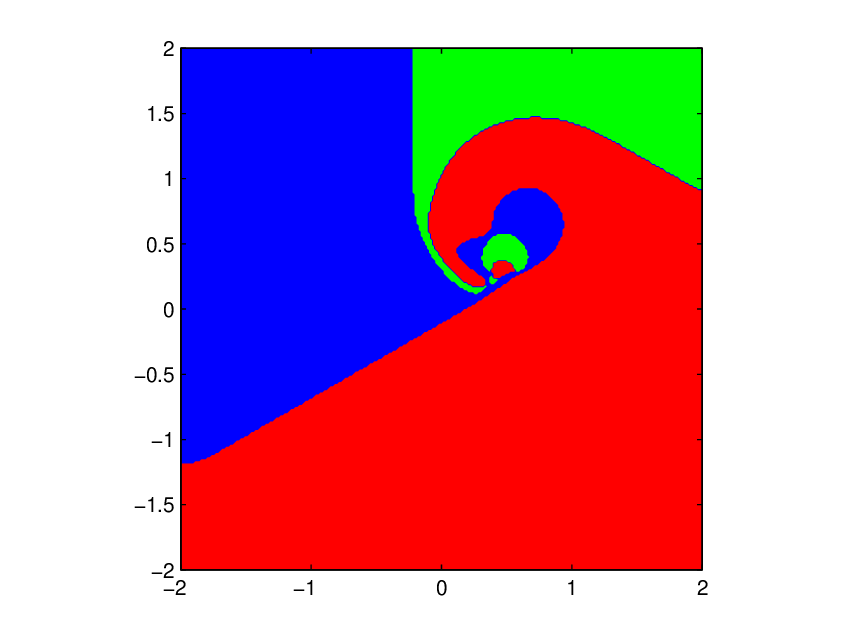}&
\includegraphics[angle=0, width=6.0cm, trim={2cm, 0.5cm, 2cm, 0}, clip, height=6.0cm]{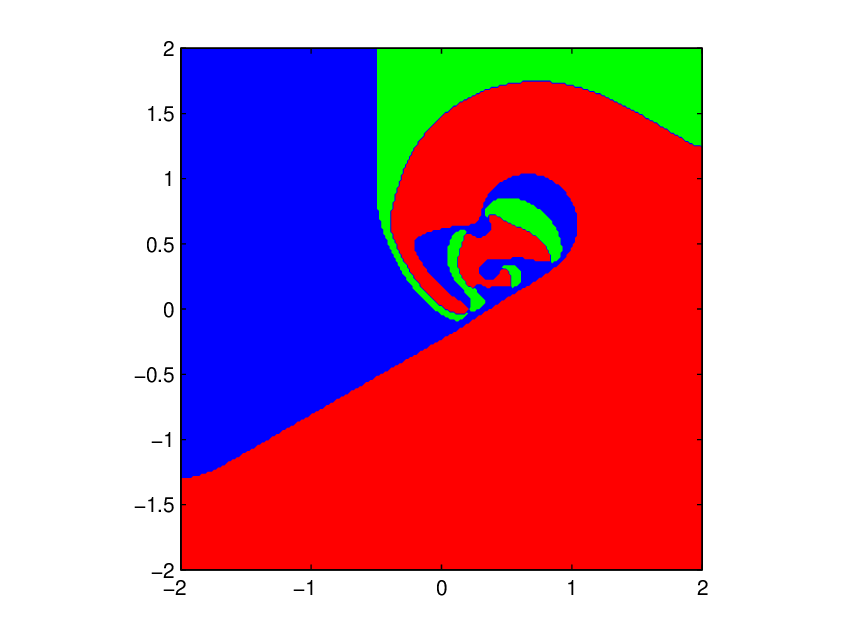}\\
$t = 150$ & $t = 200$\\
\includegraphics[angle=0, width=6.0cm, trim={2cm, 0.5cm, 2cm, 0}, clip, height=6.0cm]{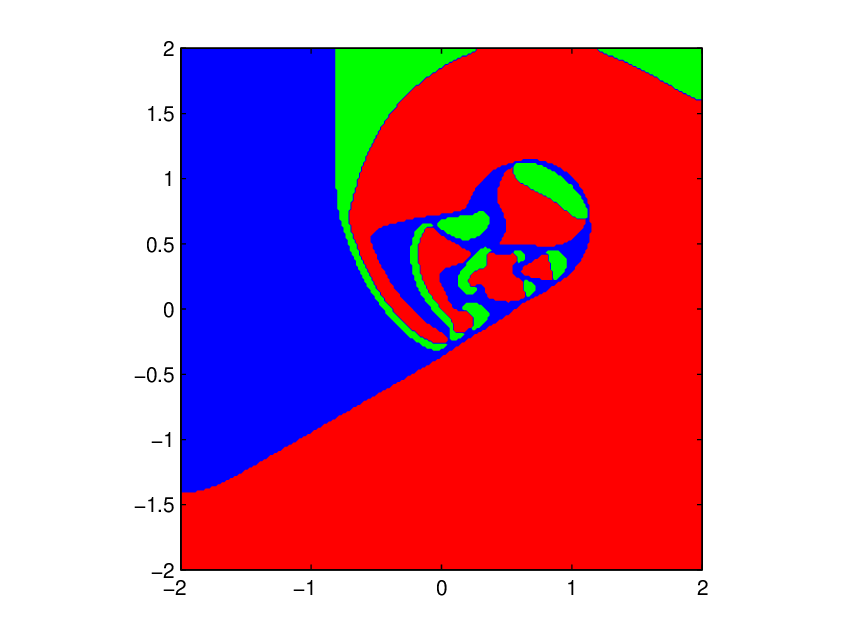}&
\includegraphics[angle=0, width=6.0cm, trim={2cm, 0.5cm, 2cm, 0}, clip, height=6.0cm]{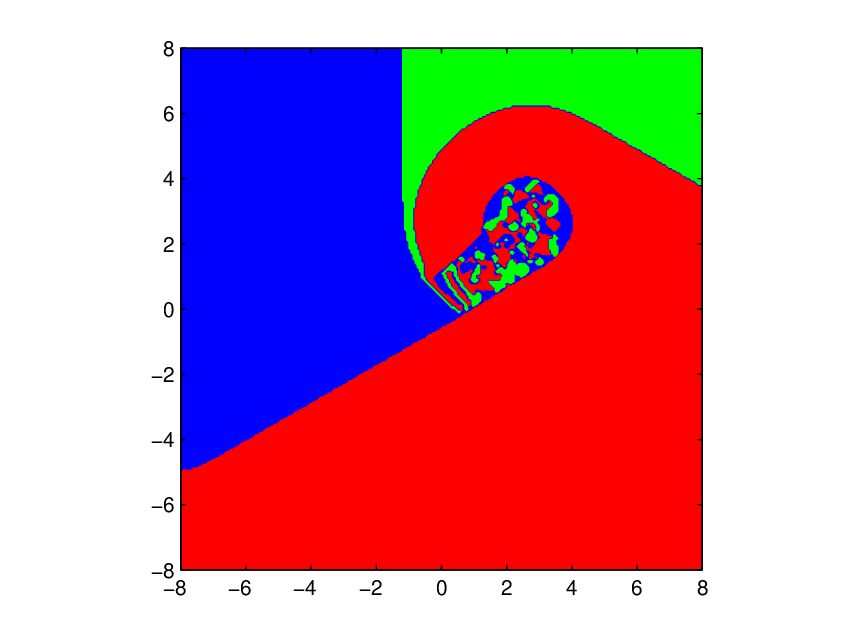}\\
$t = 250$ & $t = 300$\\
\includegraphics[angle=0, width=6.0cm, trim={2cm, 0.5cm, 2cm, 0}, clip, height=6.0cm]{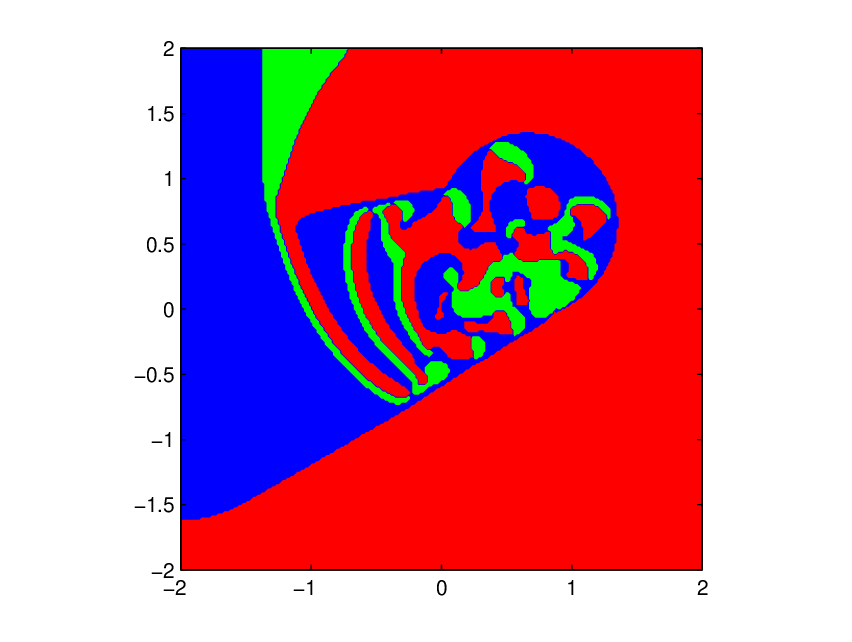}&
\includegraphics[angle=0, width=6.0cm, trim={2cm, 0.5cm, 2cm, 0}, clip, height=6.0cm]{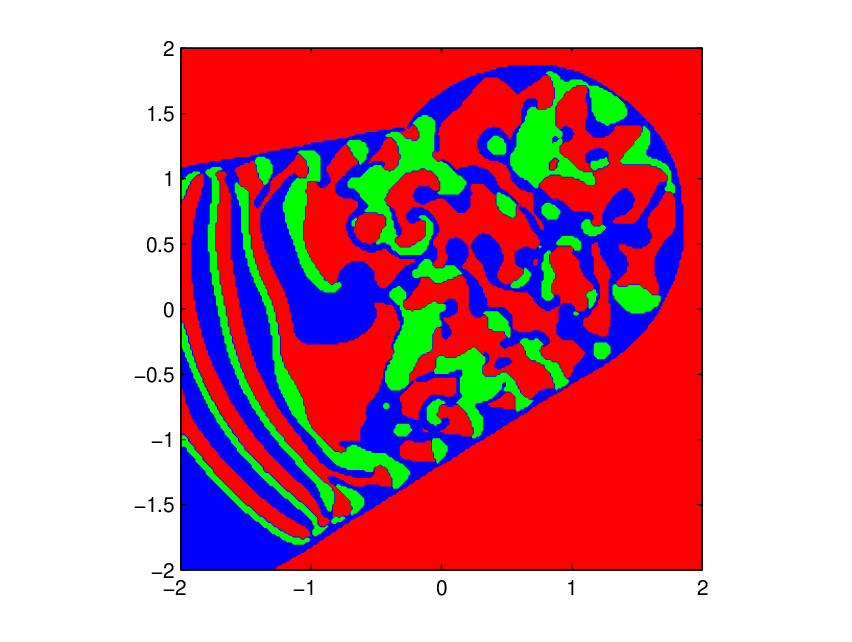}\\
$t = 350$ & $t = 600$\\
\end{tabular}
\caption{{\small Strip-like pattern in the dynamics of the three species at different times. Parameters used for the simulation are $a = 1 , ~ b = 2,~ \varepsilon_2 = 0.1,~ \varepsilon_3 = 0.9$, and $\Delta t= 1$.}}
\label{fig:Band}
\end{center}
\end{figure}
\begin{figure}[!h]
\begin{center}
\begin{tabular}{cc}

\includegraphics[angle=0, width=6.0cm, trim={2cm, 0.5cm, 2cm, 0}, clip, height=6.0cm]{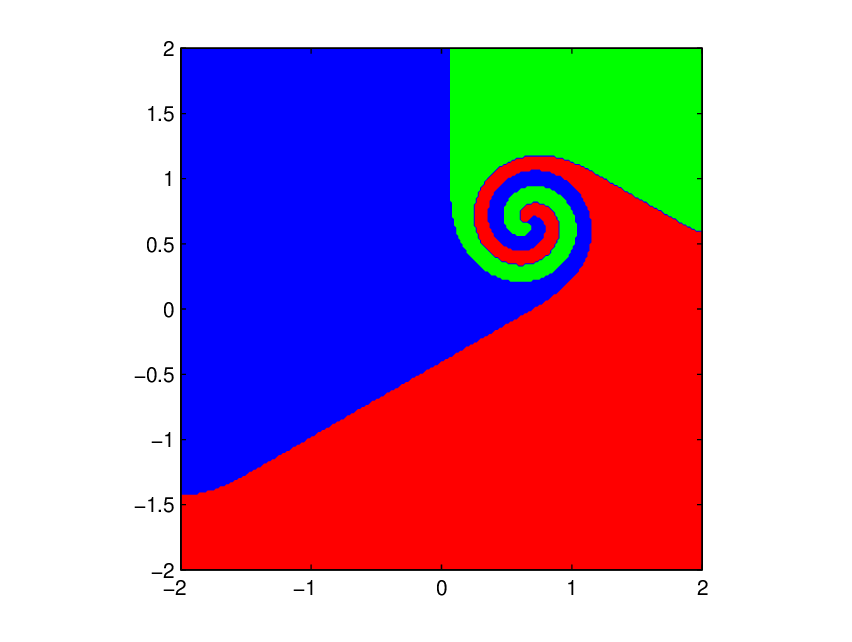}&
\includegraphics[angle=0, width=6.0cm, trim={2cm, 0.5cm, 2cm, 0}, clip, height=6.0cm]{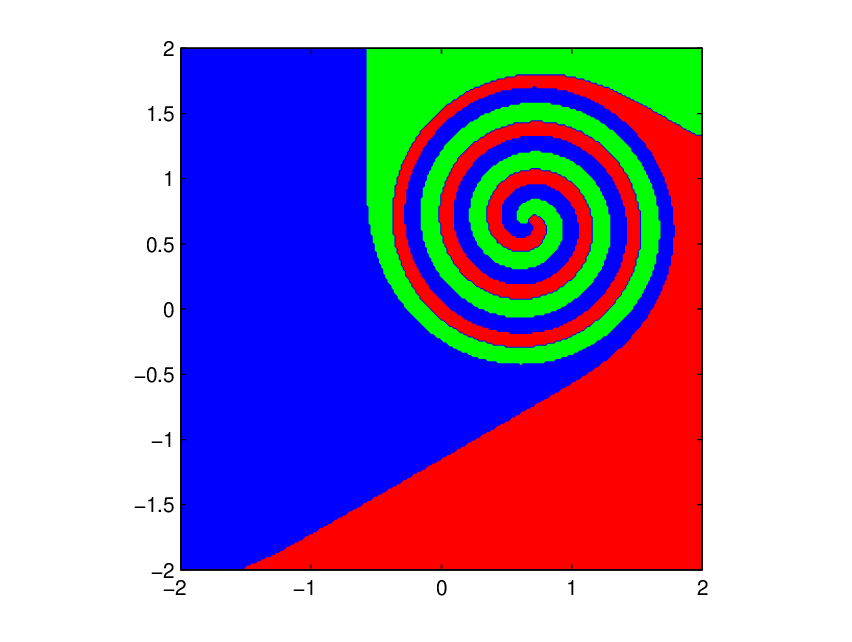}\\
$t = 0$ & $t = 10$\\
\includegraphics[angle=0, width=6.0cm, trim={2cm, 0.5cm, 2cm, 0}, clip, height=6.0cm]{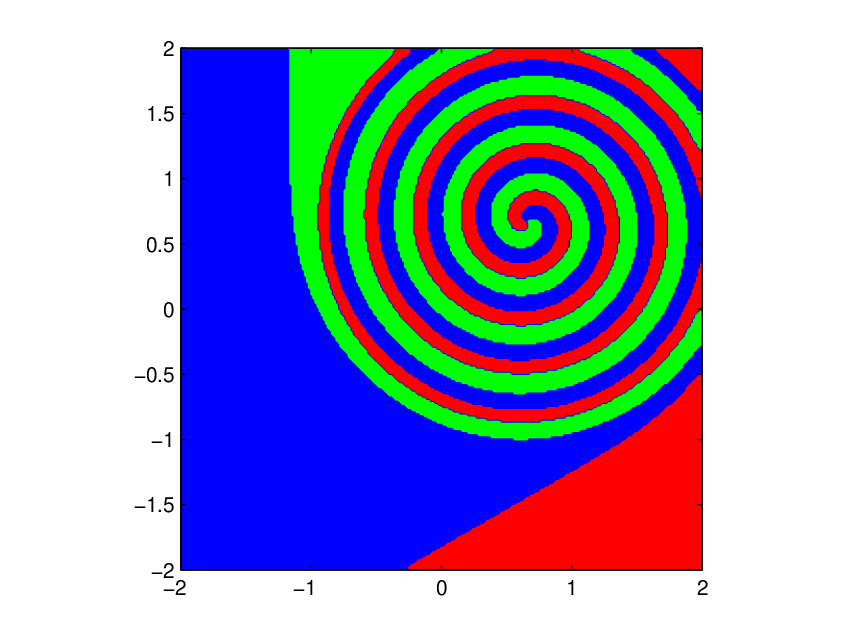}&
\includegraphics[angle=0, width=6.0cm, trim={2cm, 0.5cm, 2cm, 0}, clip, height=6.0cm]{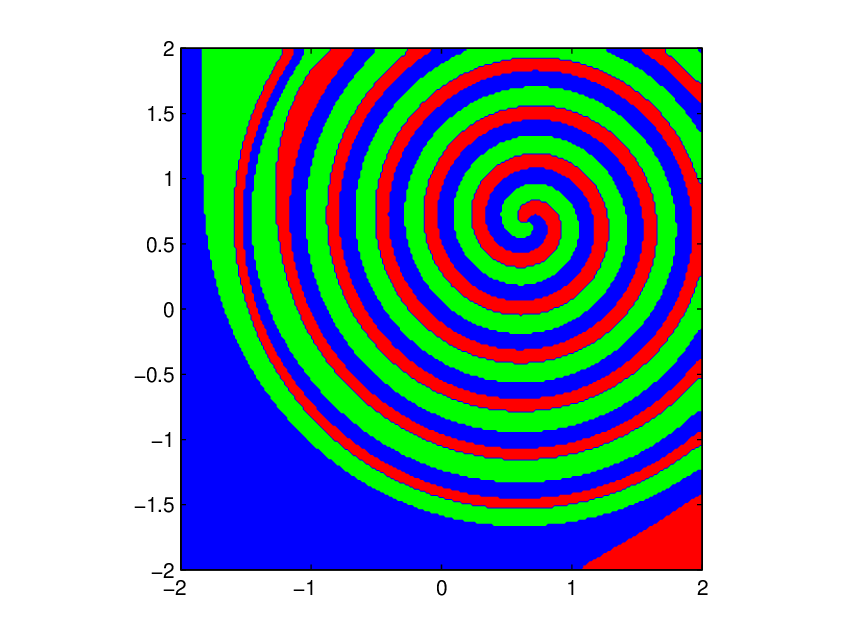}\\
$t = 25$ & $t = 38$\\
\includegraphics[angle=0, width=6.0cm, trim={2cm, 0.5cm, 2cm, 0}, clip, height=6.0cm]{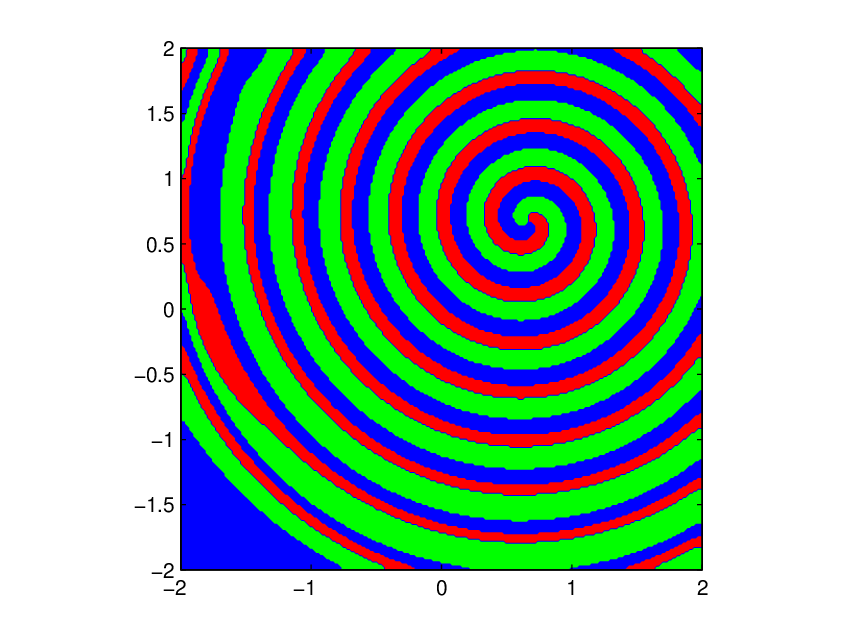}&
\includegraphics[angle=0, width=6.0cm, trim={2cm, 0.5cm, 2cm, 0}, clip, height=6.0cm]{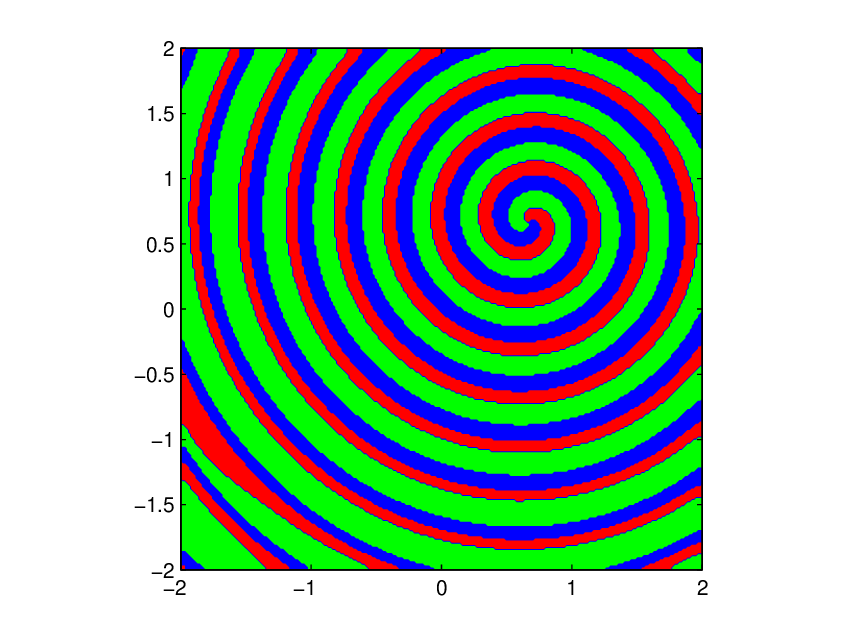}\\
$t = 50$ & $t = 63$\\
\end{tabular}
\caption{{\small Spiral-like pattern in the dynamics of the three species at different times. Parameters used for the simulation are $a = 1 , ~ b = 2,~ \varepsilon_2 = 1.0,~ \varepsilon_3 = 1.0$, and $\Delta t= 1$.}}
\label{fig:Spiral}
\end{center}
\end{figure}

\begin{figure}[!h]
\begin{center}
\begin{tabular}{ccc}
\includegraphics[angle=0, width=6.0cm, trim={2cm, 0.5cm, 2cm, 0}, clip, height=6.0cm]{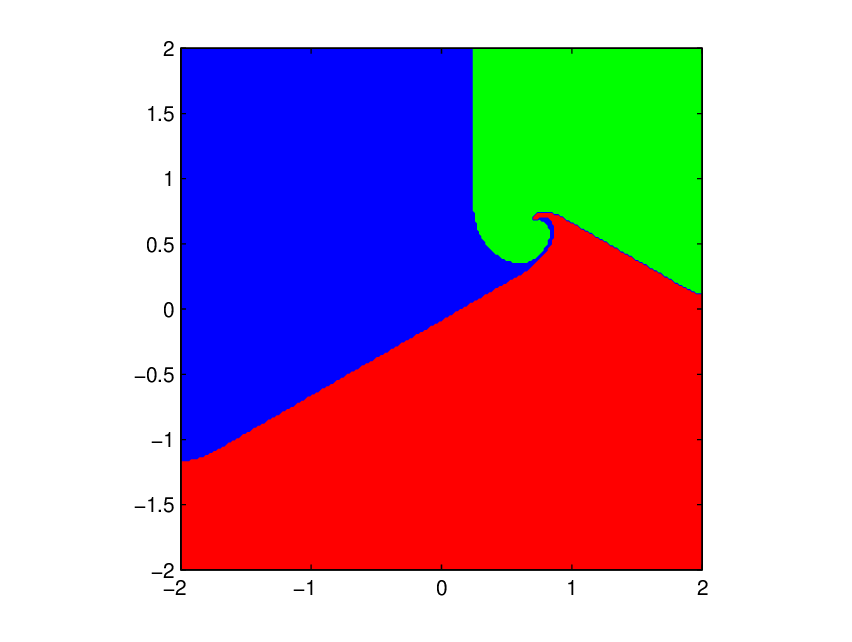}&
\includegraphics[angle=0, width=6.0cm, trim={2cm, 0.5cm, 2cm, 0}, clip, height=6.0cm]{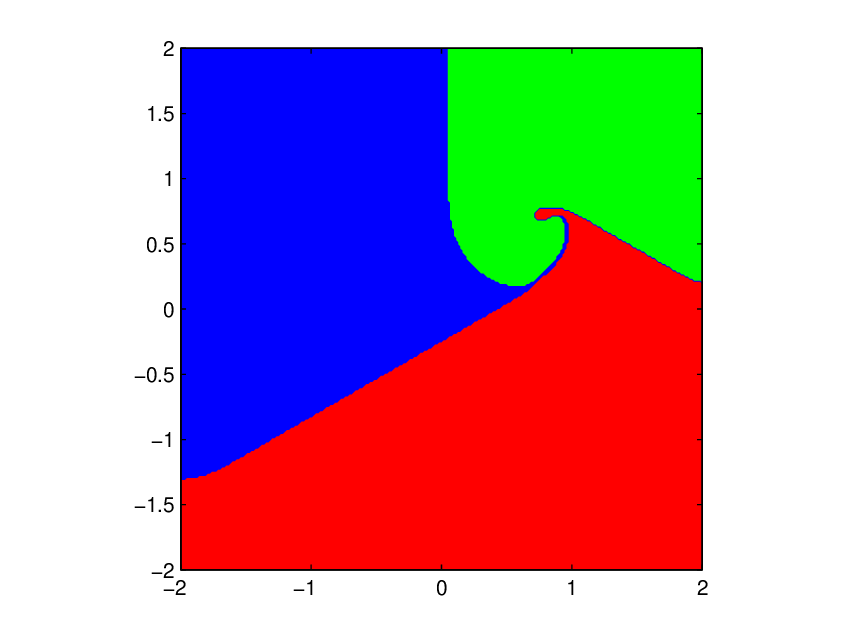}\\
$t = 70$ & $t = 100$\\
\includegraphics[angle=0, width=6.0cm, trim={2cm, 0.5cm, 2cm, 0}, clip, height=6.0cm]{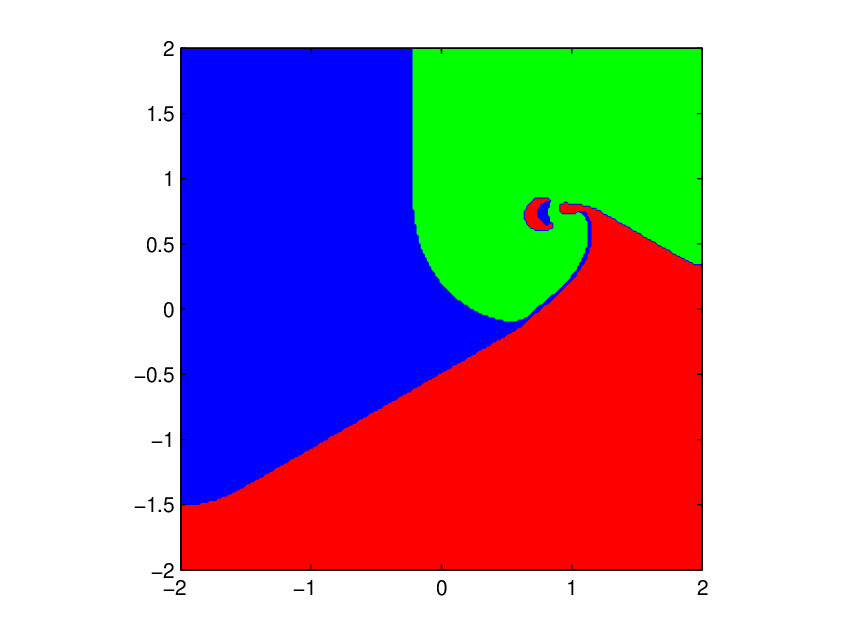}&
\includegraphics[angle=0, width=6.0cm, trim={2cm, 0.5cm, 2cm, 0}, clip, height=6.0cm]{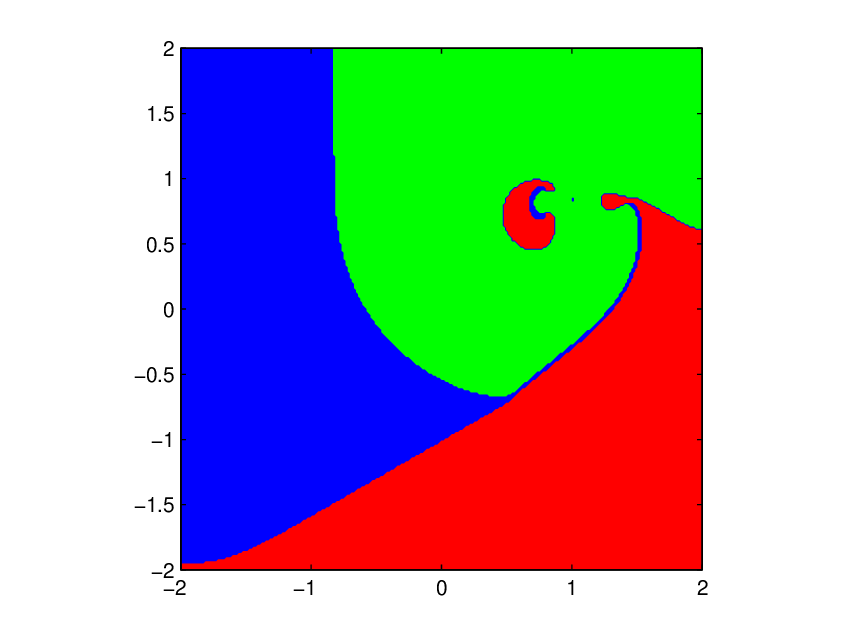}\\
$t = 150$ & $t = 250$\\
\includegraphics[angle=0, width=6.0cm, trim={2cm, 0.5cm, 2cm, 0}, clip, height=6.0cm]{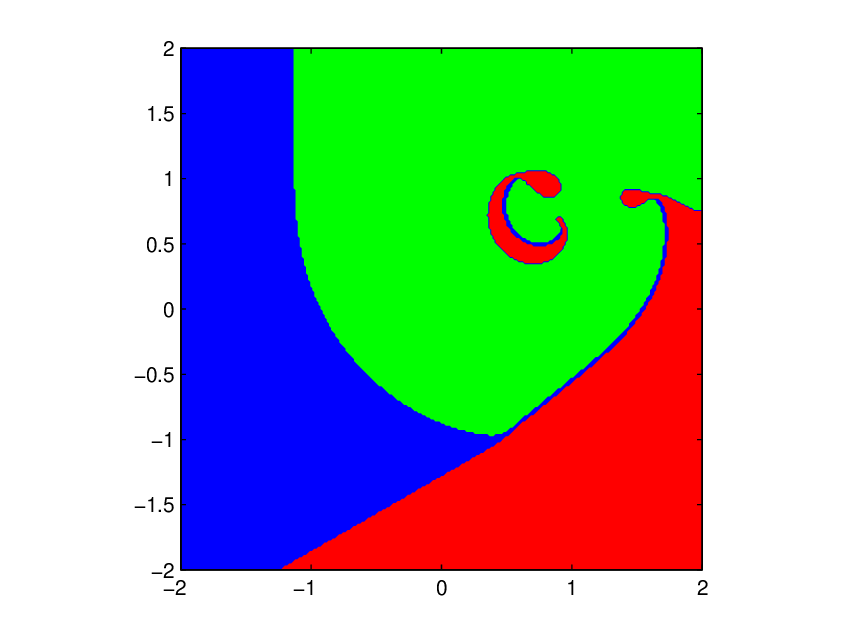}&
\includegraphics[angle=0, width=6.0cm, trim={2cm, 0.5cm, 2cm, 0}, clip, height=6.0cm]{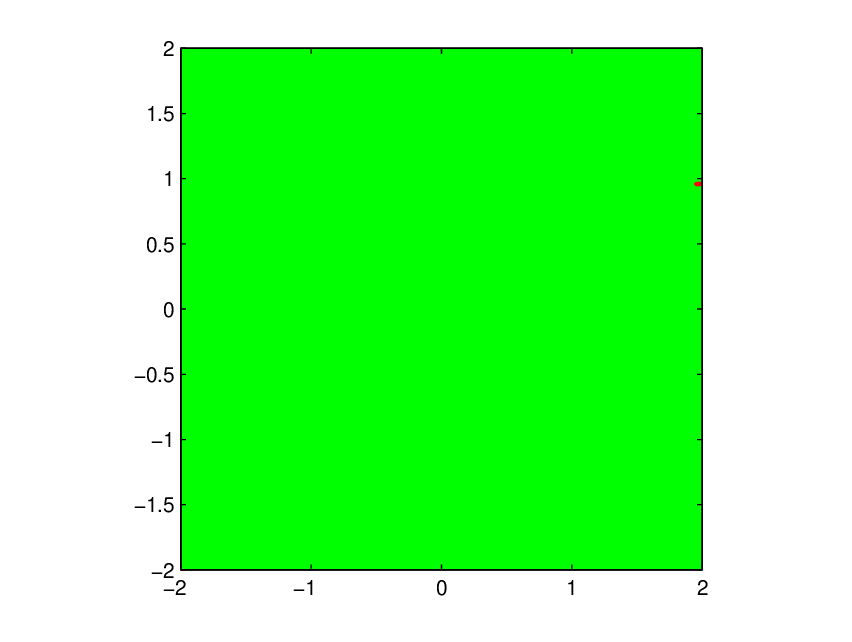}\\
$t = 300$ & $t = 400$\\
\end{tabular}
\caption{{\small Glider-like patterns of the three species at different times. Parameters used for the simulation are $a = 1 ,~ b = 2,~ \varepsilon_2 =0.55 ,~ \varepsilon_3 =0.5,~\alpha = 1.3 $, and $\Delta t = 1$.}}
\label{fig:Shooting}
\end{center}
\end{figure}

\section{Concluding remarks}\label{sec:conclusion}\label{Concluding remarks-PP-6}
In this study, we developed high-order semi-implicit multistep schemes based on the Crank-Nicolson and Adams-Bashforth methods for 
temporal discretization in conjunction with $C^0$-conforming finite element method for the nonlinear singularly perturbed three-species 
competition-diffusion model in a two-dimensional spatial domain. The semi-implicit scheme is second-order accurate in time, and has very good stability. 
Moreover, the proposed scheme has better stability than IMEX-based methods for singularly competitive-diffusion problems, 
in which diffusion is significantly less dominant than the reaction term. Several types of two-dimensional spatio-temporal patterns arising from the fact that 
the species have different mobilities, are simulated to demonstrate the performance of the proposed scheme.
\\
\\

\section{Declarations}

{\bf Funding.} KP was supported by the University of Missouri South African Education Program (UMSAEP) and South African National Research Foundation. \\

{\bf Conflicts of interest.} The authors declare that they have no conflicts of interest. \\

{\bf Availability of data and materials.} The datasets generated and analyzed during the current study are available from the corresponding author upon reasonable request. \\


{\bf Authors' contributions.} XL developed the finite element scheme, performed the numerical analysis, and wrote the manuscript. KP contributed to the design of the project, development of the multi-step numerical scheme, performed the stability analysis, and wrote the manuscript. AM performed the simulations, analyzed the numerical results, and wrote the manuscript. All authors read and approved the final manuscript.


\end{document}